\RequirePackage[l2tabu, orthodox]{nag}

\documentclass[12pt,reqno]{amsart}
\usepackage{fullpage,url,amssymb,enumerate,colonequals}
\usepackage{mathrsfs} 
\usepackage{appendix}
\usepackage{float} 
\usepackage[dvipsnames,xcdraw,hyperref]{xcolor}

\newcommand{\N}{\mathbb{N}}

\newcommand{\Q}{\mathbb{Q}}

\newcommand{\R}{\mathbb{R}}
\newcommand{\Z}{\mathbb{Z}}

\newcommand{\calD}{\mathcal{D}}

\newcommand{\M}{\mathscr{M}}

\newcommand{\tensor}{\otimes} 

\newtheorem{theorem}{Theorem}[section]
\newtheorem{lemma}[theorem]{Lemma}
\newtheorem{corollary}[theorem]{Corollary}
\newtheorem{proposition}[theorem]{Proposition}

\theoremstyle{definition}
\newtheorem{definition}[theorem]{Definition}

\theoremstyle{remark}
\newtheorem{remark}[theorem]{Remark}

\usepackage{scalerel,stackengine}
\stackMath
\newcommand\reallywidehat[1]{%
\savestack{\tmpbox}{\stretchto{%
  \scaleto{%
    \scalerel*[\widthof{\ensuremath{#1}}]{\kern.1pt\mathchar"0362\kern.1pt}%
    {\rule{0ex}{\textheight}}
  }{\textheight}%
}{2.4ex}}%
\stackon[-8.5pt]{#1}{\tmpbox}%
}

\makeatletter
\g@addto@macro\bfseries{\boldmath} 
\makeatother

\usepackage{tikz}

\usepackage[
	pagebackref,
	pdfauthor={}, 
	pdftitle={Tetrahedra Tiling Problem},
]{hyperref}

\makeatletter
\@namedef{subjclassname@2020}{%
  \textup{2020} Mathematics Subject Classification}
\makeatother

\let\int\relax
\DeclareMathOperator{\int}{int}

\DeclareMathOperator{\bd}{bd}

\newcommand{\Ti}{(T_{i})_{i}}

\renewcommand{\Xi}{(X_{i})_{i}}

\renewcommand{\xi}{X_{i}}

\newcommand{\bs}{\backslash}

\newcommand{\mb}{\mathbb}
\newcommand{\mbf}{\mathbf}

\newcommand{\msc}{\mathscr}

\newcommand{\1}{\mathbf{1}}
\DeclareMathOperator{\len}{len}
\newcommand{\E}{\mathscr{E}}

\newcommand{\x}{\mathbf{x}}
\newcommand{\y}{\mathbf{y}}

\begin{document}

\title{Tetrahedra Tiling Problem}
\subjclass[2020]{Primary  05B45; Secondary: 51M20}

\keywords{Dehn invariant, Tiling}

\author{A. Anas Chentouf}
\address{Department of Mathematics, Massachusetts Institute of Technology, Cambridge, MA 02139-4307, USA}
\email{chentouf@mit.edu}

\author{Yihang Sun}
\address{Department of Mathematics, Massachusetts Institute of Technology, Cambridge, MA 02139-4307, USA}
\email{kimisun@mit.edu}



\thanks{This research was supported in part by National Science Foundation grant DMS-1601946.}

\date{December 4, 2023}

\begin{abstract}
Kedlaya, Kolpakov, Poonen, and Rubinstein classified tetrahedra all of whose dihedral angles are rational multiples of $\pi$ into two one-parameter families (a Hill family and a new family) and $59$ sporadic tetrahedra. In this paper, we consider which of them tile space; we show that every member of the Hill family, exactly one member of the new family, and at most $40$ sporadic tetrahedra tile space.
As a corollary, we disprove the converse of Debrunner's theorem, showing that not all Dehn invariant zero tetrahedra tile space. 
\end{abstract}

\maketitle
\section{Introduction}
\subsection{History}
The Dehn invariant of a polyhedra (see Section~\ref{sec:notations}) was introduced by Max Dehn in 1901 to prove Hilbert's Third Problem that not all polyhedra with equal volume could be dissected into each other \cite{dehn1901ueber}. It is invariant under \emph{scissors-cutting}, where a polyhedron $P$ is sliced into smaller polyhedra using planes and these slices are rearranged to form polyhedron $P'$. In this case, the two polyhedra $P$ and $P'$ are said to be \emph{scissors-congruent}. 
In 1901, Dehn proved that two scissors-congruent polyhedra have equal volume and Dehn invariant \cite{dehn1901ueber}. The converse was shown by Sydler in 1965 \cite{MR0192407}. 

The $\Q$-span $V\subset \R/\Q\pi$ of the six dihedral angles of a Dehn invariant zero tetrahedron has dimension $d\in\{0, 1, \dots, 6\}$. When $d=6$, by linear independence, the vanishing of the Dehn invariant implies that all the edge lengths $e_{ij}$ are $0$; there are no such non-degenerate tetrahedra. In a separate upcoming work, we study the cases $d=2,5$, and obtain partial classifications.
When $d=0$, all six dihedral angles lie in $\mb{Q}\pi$. A classification of these \emph{rational tetrahedra} was conjectured by Poonen and Rubinstein in $1995$ and proven in $2020$ by Kedlaya, Kolpakov, Poonen and Rubinstein \cite{kedlaya2020space}. It turns out that there are two one-parameter families (one of which is discovered by \cite{kedlaya2020space}) and $59$ sporadic tetrahedra. 

The goal of this paper, not fully attained, is to determine which rational tetrahedra tile space.  A polyhedron \emph{tiles} if one can place congruent copies of it in $\mathbb{R}^3$ so that every point lies in the interior of exactly one copy or at the intersection of boundaries of two or more copies. 
The tiling problem, dating back to Aristotle \cite{doi:10.1080/0025570X.1981.11976933}, 
is closely related to the Dehn invariant: Debrunner showed that tiling polyhedra have Dehn invariant zero \cite{MR604258}.
This result does not show that any rational tetrahedra fails to tile space since all of them certainly have Dehn invariant zero.
We disprove the converse and provide further results on tiling.
\subsection{Notations}\label{sec:notations}
We use the tetrahedra notations in \cite{wirth2014relations,kedlaya2020space}. Throughout the paper, we rescale all tetrahedra to have unit volume and label the vertices as $\lbrace 1, 2, 3, 4\rbrace$.
For distinct $i, j\in \{1, 2, 3, 4\}$, let $e_{ij}$ be the length of the edge $ij$, and let $\alpha_{ij}$ be the dihedral angle along that edge.
We always list dihedral angles and edge lengths  with subscripts in the order $(12, 34, 13, 24, 14, 23)$. The \emph{Dehn invariant} of a tetrahedron is given by
\[
     \calD := \sum_{1\le i<j\le 4} e_{ij} \tensor \alpha_{ij} \; \in \; \R \tensor_{\Q} (\R/\Q\pi).
\]
Let $B(r, p)$ denote the closed Euclidean ball of radius $r\ge 0$ at $p\in \mb{R}^3$. 

\subsection{Main results}\label{sec:main}
We partially classify whether the rational tetrahedra tile space.
We let $\msc{S}\subset \msc{M}$ be the set of $59$ sporadic examples in \cite[Table 3]{kedlaya2020space} and let $ \msc{F}_1, \msc{F}_2\subset \msc{M}$ be the two one-parameter families \cite[Theorem 1.8]{kedlaya2020space}, which are parametrized as follows:
\begin{equation}\label{E:F1}
  \msc{F}_1 := \left\{ \left(\frac{\pi}{2}, \frac{\pi}{2}, \pi-2x,\frac{\pi}{3}, x, x \right): x\in \left(\frac{\pi}{6}, \frac{\pi}{2}\right)\cap \Q\pi\right\},
 \end{equation}
\begin{equation}\label{E:F2}
  \msc{F}_2 := \left\{ T_{x}:=\left(\frac{5\pi}{6} - x , \frac{\pi}{6}+x, \frac{2\pi}{3} -x, \frac{2\pi}{3}-x, x, x\right): x\in \left(\frac{\pi}{6}, \frac{\pi}{3}\right]\cap \Q\pi\right\}.
 \end{equation}

\begin{theorem}\label{thm:main-family}
The tetrahedra in $\msc{F}_1\cup \msc{F}_2$ that tile space are precisely $\msc{F}_1$.
\end{theorem}
\begin{theorem}\label{thm:main-sporadic}
The sporadic tetrahedra that tile space form a subset of the following set $\msc{A}$ of $40$ tetrahedra,
listed by their dihedral angles $\left(\alpha_{12}, \alpha_{34}, \alpha_{13}, \alpha_{24}, \alpha_{14}, \alpha_{23}\right)$ as multiples of $\pi$. 
{\tiny\begin{multline*}\msc{A}=
\Bigg\lbrace
\left( \frac{ 1 }{ 4 } , \frac{ 1 }{ 3 } , \frac{ 1 }{ 3 } , \frac{ 1 }{ 4 } , \frac{ 1 }{ 2 } , \frac{ 2 }{ 3 } \right), \left( \frac{ 5 }{ 24 } , \frac{ 3 }{ 8 } , \frac{ 1 }{ 3 } , \frac{ 1 }{ 4 } , \frac{ 13 }{ 24 } , \frac{ 5 }{ 8 } \right), \left( \frac{ 1 }{ 4 } , \frac{ 1 }{ 2 } , \frac{ 1 }{ 2 } , \frac{ 1 }{ 3 } , \frac{ 1 }{ 3 } , \frac{ 1 }{ 2 } \right), \left( \frac{ 7 }{ 24 } , \frac{ 11 }{ 24 } , \frac{ 13 }{ 24 } , \frac{ 7 }{ 24 } , \frac{ 1 }{ 3 } , \frac{ 1 }{ 2 } \right), \left( \frac{ 1 }{ 5 } , \frac{ 1 }{ 5 } , \frac{ 1 }{ 3 } , \frac{ 1 }{ 5 } , \frac{ 2 }{ 3 } , \frac{ 2 }{ 3 } \right),  
\\ \left( \frac{ 1 }{ 5 } , \frac{ 1 }{ 5 } , \frac{ 4 }{ 15 } , \frac{ 4 }{ 15 } , \frac{ 3 }{ 5 } , \frac{ 11 }{ 15 } \right), \left( \frac{ 1 }{ 5 } , \frac{ 1 }{ 5 } , \frac{ 1 }{ 3 } , \frac{ 1 }{ 3 } , \frac{ 3 }{ 5 } , \frac{ 3 }{ 5 } \right), \left( \frac{ 1 }{ 3 } , \frac{ 1 }{ 3 } , \frac{ 3 }{ 5 } , \frac{ 1 }{ 3 } , \frac{ 2 }{ 5 } , \frac{ 2 }{ 5 } \right), \left( \frac{ 4 }{ 15 } , \frac{ 8 }{ 15 } , \frac{ 1 }{ 3 } , \frac{ 1 }{ 3 } , \frac{ 7 }{ 15 } , \frac{ 7 }{ 15 } \right), \left( \frac{ 1 }{ 7 } , \frac{ 3 }{ 7 } , \frac{ 1 }{ 3 } , \frac{ 1 }{ 3 } , \frac{ 4 }{ 7 } , \frac{ 4 }{ 7 } \right),  
\\ \left( \frac{ 2 }{ 7 } , \frac{ 2 }{ 7 } , \frac{ 1 }{ 3 } , \frac{ 1 }{ 3 } , \frac{ 3 }{ 7 } , \frac{ 5 }{ 7 } \right), \left( \frac{ 1 }{ 5 } , \frac{ 2 }{ 5 } , \frac{ 1 }{ 2 } , \frac{ 1 }{ 3 } , \frac{ 1 }{ 3 } , \frac{ 2 }{ 3 } \right), \left( \frac{ 2 }{ 15 } , \frac{ 7 }{ 15 } , \frac{ 1 }{ 2 } , \frac{ 1 }{ 3 } , \frac{ 2 }{ 5 } , \frac{ 3 }{ 5 } \right), \left( \frac{ 2 }{ 15 } , \frac{ 7 }{ 15 } , \frac{ 31 }{ 60 } , \frac{ 19 }{ 60 } , \frac{ 5 }{ 12 } , \frac{ 7 }{ 12 } \right), \left( \frac{ 1 }{ 5 } , \frac{ 2 }{ 5 } , \frac{ 7 }{ 12 } , \frac{ 1 }{ 4 } , \frac{ 5 }{ 12 } , \frac{ 7 }{ 12 } \right), 
\\  \left( \frac{ 13 }{ 60 } , \frac{ 23 }{ 60 } , \frac{ 7 }{ 12 } , \frac{ 1 }{ 4 } , \frac{ 2 }{ 5 } , \frac{ 3 }{ 5 } \right), \left( \frac{ 1 }{ 5 } , \frac{ 3 }{ 5 } , \frac{ 1 }{ 3 } , \frac{ 1 }{ 3 } , \frac{ 1 }{ 2 } , \frac{ 1 }{ 2 } \right), \left( \frac{ 3 }{ 10 } , \frac{ 7 }{ 10 } , \frac{ 1 }{ 3 } , \frac{ 1 }{ 3 } , \frac{ 2 }{ 5 } , \frac{ 2 }{ 5 } \right), \left( \frac{ 1 }{ 5 } , \frac{ 1 }{ 5 } , \frac{ 2 }{ 5 } , \frac{ 1 }{ 3 } , \frac{ 1 }{ 2 } , \frac{ 2 }{ 3 } \right), \left( \frac{ 1 }{ 6 } , \frac{ 7 }{ 30 } , \frac{ 11 }{ 30 } , \frac{ 11 }{ 30 } , \frac{ 1 }{ 2 } , \frac{ 2 }{ 3 } \right), \\
\left( \frac{ 1 }{ 5 } , \frac{ 1 }{ 5 } , \frac{ 9 }{ 20 } , \frac{ 17 }{ 60 } , \frac{ 11 }{ 20 } , \frac{ 37 }{ 60 } \right), \left( \frac{ 1 }{ 5 } , \frac{ 1 }{ 3 } , \frac{ 1 }{ 2 } , \frac{ 1 }{ 3 } , \frac{ 2 }{ 5 } , \frac{ 3 }{ 5 } \right), \left( \frac{ 1 }{ 6 } , \frac{ 11 }{ 30 } , \frac{ 1 }{ 2 } , \frac{ 1 }{ 3 } , \frac{ 13 }{ 30 } , \frac{ 17 }{ 30 } \right), \left( \frac{ 1 }{ 6 } , \frac{ 11 }{ 30 } , \frac{ 29 }{ 60 } , \frac{ 7 }{ 20 } , \frac{ 5 }{ 12 } , \frac{ 7 }{ 12 } \right), \left( \frac{ 1 }{ 5 } , \frac{ 1 }{ 3 } , \frac{ 1 }{ 3 } , \frac{ 1 }{ 5 } , \frac{ 1 }{ 2 } , \frac{ 4 }{ 5 } \right),  
\\\left( \frac{ 7 }{ 60 } , \frac{ 5 }{ 12 } , \frac{ 1 }{ 3 } , \frac{ 1 }{ 5 } , \frac{ 7 }{ 12 } , \frac{ 43 }{ 60 } \right), \left( \frac{ 1 }{ 5 } , \frac{ 2 }{ 5 } , \frac{ 2 }{ 5 } , \frac{ 1 }{ 5 } , \frac{ 1 }{ 2 } , \frac{ 2 }{ 3 } \right), \left( \frac{ 1 }{ 5 } , \frac{ 2 }{ 5 } , \frac{ 1 }{ 3 } , \frac{ 1 }{ 3 } , \frac{ 1 }{ 2 } , \frac{ 3 }{ 5 } \right),  \left( \frac{ 7 }{ 30 } , \frac{ 13 }{ 30 } , \frac{ 3 }{ 10 } , \frac{ 3 }{ 10 } , \frac{ 1 }{ 2 } , \frac{ 3 }{ 5 } \right), \left( \frac{ 1 }{ 4 } , \frac{ 7 }{ 20 } , \frac{ 1 }{ 3 } , \frac{ 1 }{ 3 } , \frac{ 9 }{ 20 } , \frac{ 13 }{ 20 } \right), \\
\left( \frac{ 1 }{ 5 } , \frac{ 1 }{ 2 } , \frac{ 3 }{ 5 } , \frac{ 1 }{ 5 } , \frac{ 1 }{ 3 } , \frac{ 2 }{ 3 } \right), \left( \frac{ 1 }{ 5 } , \frac{ 1 }{ 2 } , \frac{ 17 }{ 30 } , \frac{ 7 }{ 30 } , \frac{ 3 }{ 10 } , \frac{ 7 }{ 10 } \right),  \left( \frac{ 3 }{ 20 } , \frac{ 11 }{ 20 } , \frac{ 17 }{ 30 } , \frac{ 7 }{ 30 } , \frac{ 7 }{ 20 } , \frac{ 13 }{ 20 } \right), \left( \frac{ 3 }{ 20 } , \frac{ 11 }{ 20 } , \frac{ 11 }{ 20 } , \frac{ 1 }{ 4 } , \frac{ 1 }{ 3 } , \frac{ 2 }{ 3 } \right), 
\\ \left( \frac{ 11 }{ 60 } , \frac{ 31 }{ 60 } , \frac{ 11 }{ 20 } , \frac{ 1 }{ 4 } , \frac{ 3 }{ 10 } , \frac{ 7 }{ 10 } \right), \left( \frac{ 1 }{ 5 } , \frac{ 1 }{ 2 } , \frac{ 1 }{ 2 } , \frac{ 1 }{ 3 } , \frac{ 2 }{ 5 } , \frac{ 1 }{ 2 } \right), \left( \frac{ 1 }{ 5 } , \frac{ 1 }{ 2 } , \frac{ 7 }{ 15 } , \frac{ 11 }{ 30 } , \frac{ 11 }{ 30 } , \frac{ 8 }{ 15 } \right), \left( \frac{ 4 }{ 15 } , \frac{ 13 }{ 30 } , \frac{ 17 }{ 30 } , \frac{ 4 }{ 15 } , \frac{ 2 }{ 5 } , \frac{ 1 }{ 2 } \right), 
\\ \left( \frac{ 4 }{ 15 } , \frac{ 13 }{ 30 } , \frac{ 3 }{ 5 } , \frac{ 3 }{ 10 } , \frac{ 11 }{ 30 } , \frac{ 7 }{ 15 } \right), \left( \frac{ 4 }{ 15 } , \frac{ 17 }{ 30 } , \frac{ 2 }{ 5 } , \frac{ 3 }{ 10 } , \frac{ 11 }{ 30 } , \frac{ 8 }{ 15 } \right), \left( \frac{ 3 }{ 10 } , \frac{ 2 }{ 5 } , \frac{ 3 }{ 5 } , \frac{ 3 }{ 10 } , \frac{ 1 }{ 3 } , \frac{ 1 }{ 2 } \right), \left( \frac{ 1 }{ 3 } , \frac{ 2 }{ 5 } , \frac{ 2 }{ 5 } , \frac{ 1 }{ 3 } , \frac{ 1 }{ 2 } , \frac{ 2 }{ 5 } \right)
\Bigg\rbrace.
\end{multline*}}
\end{theorem}
Recall that Debrunner's theorem states that tiling polyhedra must have Dehn invariant zero. We immediately obtain that negative of its converse.
\begin{corollary}
Not all Dehn invaraint zero tetrahedra tile space.
\end{corollary}
While $\msc{F}_2$ is a new discovery by \cite{kedlaya2020space}, $\msc{F}_1$ is the first Hill family and all of its members tile space \cite{MR1576480,goldberg1974three}. Moreover, as the ``No. (1) tetrahedron'' in \cite{sommerville1922space}, $T_{\pi/3}\in\msc{F}_1$, so it tiles space. Thus, Theorems~\ref{thm:main-family} and \ref{thm:main-sporadic} follow the following lemma, shown in Section~\ref{sec:app}. To prove it, we develop in Section~\ref{sec:criteria} some necessary conditions for a tetrahedron to tile space.
\begin{lemma}\label{lem:negative}
None of the tetrahedra in $\msc{F}_2\bs\{T_{\pi/3}\}$ and $\msc{S}\setminus\msc{A}$ tile space.
\end{lemma}
\subsection*{Acknowledgements}
The authors would like to extend our profound gratitude towards Professor Bjorn Poonen for advising this project and for many insightful discussions.

\section{Criteria for Tiling Tetrahedra}\label{sec:criteria}
\begin{definition}\label{def:gen-def}
We first establish some terminology.
\begin{itemize}
    \item For non-zero real numbers $a,b$, we say that $a$ divides $b$, written $a|b,$ if the ratio $\frac{b}{a}\in\Z$.
    \item Two edges of tetrahedron are \emph{incident} if they share a vertex and \emph{opposite} otherwise.
    \item We use \emph{combination} to mean nonnegative integer combinations. An element is \emph{involved} in a combination if it has positive coefficient as opposed to $0$.Define a \emph{dihedral $\pi$-combination} as a combination of dihedral angles of a tetrahedron that gives $\pi$.
\end{itemize}
\end{definition}
Note that every tetrahedron has $3$ pairs of opposite edges and $12$ pairs of incident edges.
The motivation for the last definition and the starting point of our criteria lies in the following simple observation: the $2\pi$ angle around an edge must be made up of dihedral angles and possibly one face that contributes $\pi$, so a combination of dihedral angles gives $\pi$ or $2\pi$. 

Motivated by this, we classify tilings according to if any face contributes, as in \cite{sommerville1922space}.
\begin{definition}\label{def:ff-nf}
A tiling by tetrahedron $T$ is \emph{face-to-face} if no edge of one copy of $T$ intersects the interior of a face of a possibly different copy of $T$. It is \emph{non-face-to-face} otherwise.
\end{definition}
\subsection{Criteria for non-face-to-face tilings}\label{sec:nf}
For non-face-to-face tilings, we present the following criteria. Heuristically, both conditions are similar to the aforementioned simple observation about combinations of angles used to make up the $2\pi$ angle around an edge.
\begin{proposition}\label{prop:crtn}
Tetrahedron $T$ has no non-face-to-face tiling if either
\begin{enumerate}
\item $T$ has no dihedral $\pi$-combination;
\item $T$ has only one dihedral $\pi$-combination and it involves only a pair of opposite edges.
\end{enumerate}
\end{proposition}
\begin{proof}
Suppose $T$ has a non-face-to-face tiling, then some edge $e$ intersects the interior of face $F'$ of another copy of $T$. 
For (1), note that the $2\pi$ angle around the part of $e$ that lies on $F'$ is made up by $\pi$ from $F'$ and dihedral angles including $e$. Counting contributions by dihedral angles gives the desired combination.

For (2), since non-opposite edges are incident, it suffices to find two incident edges that both lie on the face of a different copy of $T$.
If endpoint $v$ of $e$ lies in $F'$, then the two edges incident to $v$ that lie on the plane containing $F'$ satisfy the condition needed for (2).

Otherwise, both endpoints of $e$ are outside $F'$. The line though $e$ on the plane containing $F'$ divides the $3$ vertices of $F'$ into two sides, both of which contains at least one vertex of $F'$. Consider the side of $e$ that has one vertex. Call the vertex $v'$.

Suppose $v'$ lies in $F$, then the two edges of $T'$ incident to $v'$ satisfy the condition needed for (2).
If $v'$ does not lie in $F$, then another edge $e_{1}$ of $T$ also intersects $\int F'$. Then, $e$ and $e'$ are coplanar edges of $T$, so they are incident. They satisfy the condition needed for (2).
\end{proof}
\begin{remark}
Proposition~\ref{prop:crtn} is implemented in \cite{criteriapy} as CRTN by generating the dihedral $\pi$-combinations that give $\pi$ or $2\pi$ and checking the conditions.
\end{remark}


\subsection{Criteria for face-to-face tilings}\label{sec:ff}
We can rule out face-to-face tilings for specific kinds of tetrahedra depending on the edge lengths. Fortunately, this is sufficient for Lemma~\ref{lem:negative}.
\begin{definition}\label{def:para}
A tetrahedron is \emph{parallelogram-like} if exactly two pairs of opposite edges are congruent and are not equal between them, i.e. the edge lengths in the usual order are of the form $(a, a, b, b, c, d), (a, b, c, c, d, d)$, or $(a, a, b, c, d, d)$. 
\end{definition}
As the name suggests, these tetrahedra arise when we fold a parallelogram along a diagonal.
\begin{lemma}\label{lem:ff-lem}
Given distinct polygons $P$ and $Q$ lying in the same plane, if $\int P\cap\int Q\neq\varnothing$, then $\bd P\cap \int Q\neq\varnothing$ or $\int P\cap \bd Q\neq\varnothing$.
\end{lemma}
\begin{proof}
Suppose otherwise, let $x\in \int P\cap \int Q$. Take any ray $\ell$ at $x$ intersecting $\bd P$ and $\bd Q$ at $p$ and $q$, respectively. Without loss of generality, assume that $\Vert p-x\Vert \leq \Vert q-x\Vert$. If the inequality is strict, then $p\in \int Q$, contradiction. Thus, it must be $p=q$. Since this holds for every ray $\ell$ at $x$, it must be $P$ and $Q$ overlap precisely, i.e. $P=Q$, a contradiction.
\end{proof}
We now present the criterion to rule out face-to-face tilings for special tetrahedra.
\begin{proposition}\label{prop:crtf}
For some tetrahedron $T$ that either is parallelogram-like or has distinct edge-lengths, $T$ has no face-to-face tiling if some dihedral angle of $T$ does not divide $2\pi$.
\end{proposition}
\begin{proof}
By Lemma~\ref{lem:ff-lem}, pairs of faces in a face-to-face tiling fit exactly, so edges fit on edges. Consider a copy of tetrahedron $T$ and face $F=(123)$. 
If $T$ has distinct edge-lengths, $T$ has no congruent faces, so $F$ must be glued precisely to the same face of another copy $T'$. Since no face is isosceles, $T'$ is obtained by reflecting $T$ over the plane containing the two faces. This holds for every face in every copy of $T$, so the tiling is generated by repeated reflections.

For parallelogram-like $T$, by Definition~\ref{def:para}, to glue onto face $F$, we can use $F$ or the face $F'=(1'2'4')$ congruent to $F$. We see in Figure~\ref{fig:para} that they are identical up to relabeling.

\tikzset{every picture/.style={line width=0.75pt}} 
\begin{figure}[H]
\centering
\begin{tikzpicture}[x=0.75pt,y=0.75pt,yscale=-1.35,xscale=1.35]

\draw [color={rgb, 255:red, 189; green, 16; blue, 224 }  ,draw opacity=1 ] [dash pattern={on 4.5pt off 4.5pt}]  (157,100) -- (237,130) ;
\draw [color={rgb, 255:red, 189; green, 16; blue, 224 }  ,draw opacity=1 ]   (227,80) -- (177,170) ;
\draw    (227,80) -- (187,70) ;
\draw [color={rgb, 255:red, 189; green, 16; blue, 224 }  ,draw opacity=1 ]   (187,70) -- (157,100) ;
\draw [color={rgb, 255:red, 80; green, 227; blue, 194 }  ,draw opacity=1 ]   (187,70) -- (177,170) ;
\draw [color={rgb, 255:red, 80; green, 227; blue, 194 }  ,draw opacity=1 ] [dash pattern={on 4.5pt off 4.5pt}]  (157,100) -- (227,80) ;
\draw [color={rgb, 255:red, 245; green, 166; blue, 35 }  ,draw opacity=1 ]   (157,100) -- (177,170) ;
\draw [color={rgb, 255:red, 80; green, 227; blue, 194 }  ,draw opacity=1 ]   (177,170) -- (237,130) ;
\draw    (227,80) -- (237,130) ;
\draw [color={rgb, 255:red, 189; green, 16; blue, 224 }  ,draw opacity=1 ] [dash pattern={on 4.5pt off 4.5pt}]  (21,100) -- (101,130) ;
\draw [color={rgb, 255:red, 189; green, 16; blue, 224 }  ,draw opacity=1 ]   (91,80) -- (41,170) ;
\draw [color={rgb, 255:red, 80; green, 227; blue, 194 }  ,draw opacity=1 ]   (21,100) -- (91,80) ;
\draw [color={rgb, 255:red, 245; green, 166; blue, 35 }  ,draw opacity=1 ]   (21,100) -- (41,170) ;
\draw [color={rgb, 255:red, 80; green, 227; blue, 194 }  ,draw opacity=1 ]   (41,170) -- (101,130) ;
\draw    (91,80) -- (101,130) ;
\draw [color={rgb, 255:red, 189; green, 16; blue, 224 }  ,draw opacity=1 ] [dash pattern={on 4.5pt off 4.5pt}]  (295,100) -- (375,130) ;
\draw [color={rgb, 255:red, 189; green, 16; blue, 224 }  ,draw opacity=1 ]   (365,80) -- (315,170) ;
\draw    (365,80) -- (325,70) ;
\draw [color={rgb, 255:red, 189; green, 16; blue, 224 }  ,draw opacity=1 ]   (325,70) -- (295,100) ;
\draw [color={rgb, 255:red, 80; green, 227; blue, 194 }  ,draw opacity=1 ]   (325,70) -- (315,170) ;
\draw [color={rgb, 255:red, 80; green, 227; blue, 194 }  ,draw opacity=1 ] [dash pattern={on 4.5pt off 4.5pt}]  (295,100) -- (365,80) ;
\draw [color={rgb, 255:red, 245; green, 166; blue, 35 }  ,draw opacity=1 ]   (295,100) -- (315,170) ;
\draw [color={rgb, 255:red, 80; green, 227; blue, 194 }  ,draw opacity=1 ]   (315,170) -- (375,130) ;
\draw    (365,80) -- (375,130) ;

\draw (228,73.4) node [anchor=north west][inner sep=0.75pt]  [font=\tiny]  {$1=1'$};
\draw (131,102.4) node [anchor=north west][inner sep=0.75pt]  [font=\tiny]  {$2=2'$};
\draw (178,173.4) node [anchor=north west][inner sep=0.75pt]  [font=\tiny]  {$3=3'$};
\draw (238,132.4) node [anchor=north west][inner sep=0.75pt]  [font=\tiny]  {$4$};
\draw (188,62.4) node [anchor=north west][inner sep=0.75pt]  [font=\tiny]  {$4'$};
\draw (150,188) node [anchor=north west][inner sep=0.75pt]  [font=\scriptsize] [align=left] {Reflection Scheme};
\draw (92,73.4) node [anchor=north west][inner sep=0.75pt]  [font=\tiny]  {$1$};
\draw (12,103.4) node [anchor=north west][inner sep=0.75pt]  [font=\tiny]  {$2$};
\draw (42,173.4) node [anchor=north west][inner sep=0.75pt]  [font=\tiny]  {$3$};
\draw (103,133.4) node [anchor=north west][inner sep=0.75pt]  [font=\tiny]  {$4$};
\draw (11,182) node [anchor=north west][inner sep=0.75pt]  [font=\scriptsize] [align=left] {$ $};
\draw (58,182.4) node [anchor=north west][inner sep=0.75pt]  [font=\scriptsize]  {$T$};
\draw (366,73.4) node [anchor=north west][inner sep=0.75pt]  [font=\tiny]  {$1=2'$};
\draw (271,102.4) node [anchor=north west][inner sep=0.75pt]  [font=\tiny]  {$2=1'$};
\draw (317,173.4) node [anchor=north west][inner sep=0.75pt]  [font=\tiny]  {$3=4'$};
\draw (376,132.4) node [anchor=north west][inner sep=0.75pt]  [font=\tiny]  {$4$};
\draw (328,62.4) node [anchor=north west][inner sep=0.75pt]  [font=\tiny]  {$3'$};
\draw (311,188) node [anchor=north west][inner sep=0.75pt]  [font=\scriptsize] [align=left] {Other Tiling};
\draw (218,149.4) node [anchor=north west][inner sep=0.75pt]  [font=\scriptsize]  {$T$};
\draw (161,72.4) node [anchor=north west][inner sep=0.75pt]  [font=\scriptsize]  {$T'$};
\draw (358,149.4) node [anchor=north west][inner sep=0.75pt]  [font=\scriptsize]  {$T$};
\draw (298,72.4) node [anchor=north west][inner sep=0.75pt]  [font=\scriptsize]  {$T'$};

\end{tikzpicture}
\caption{Two ways to tile parallelogram-like tetrahedra.}
\label{fig:para}
\end{figure}
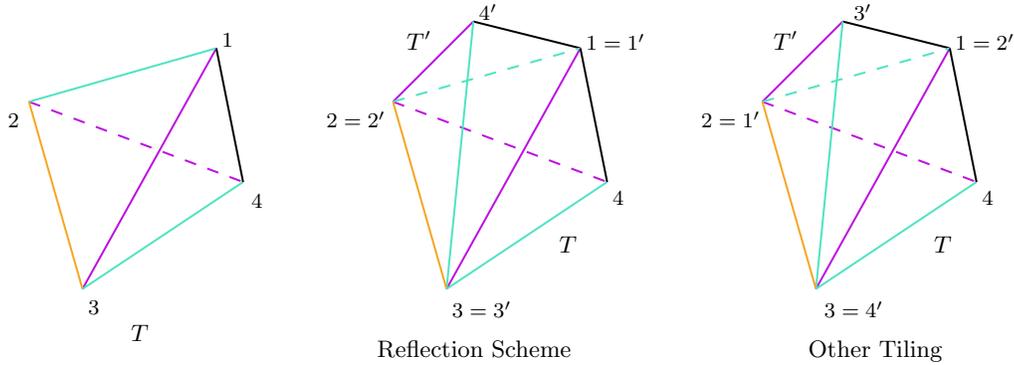

In both cases, it suffices to show that repeated reflections cannot generate a tiling. Take any edge $e$ and let its dihedral angle be $\alpha$. Under the reflection schedule, the dihedral angles around $e$ must are copies of the same angle with measure $\alpha$, so $\alpha\mid 2\pi$ for every edge $e$.
\end{proof}
\begin{remark}
Proposition~\ref{prop:crtf} is implemented in \cite{criteriapy} as CRTF by computing the edge lengths and dihedral angles and checking the conditions. Note that edge lengths of a tetrahedron can be computed from its dihedral angles via \cite{wirth2014relations}, as implemented in \cite{tetrapy}.
\end{remark}
\subsection{Criteria from linear programming}\label{sec:lp}
Let $\M$ be the moduli space of unit-volume tetrahedra with labeled vertices, up to congruence preserving the labeling.
Fix any tiling of $\mb{R}^3$ using shape $T\in\msc{M}$.
Define the diameter $\gamma = \max \{\Vert x-y\Vert:x, y\in T\}>0$ by the extreme value theorem.
We parametrize the tiling as follows. Since $\mb{Q}^3$ is dense in $\mb{R}^3$, each copy of $T$ in the tiling contains a rational point, so there are countable many copies of $T$. We enumerate them as $T_1, T_2, \dots$ in increasing order of $R_{i} := \max\{\Vert x\Vert:x\in T_i\}$. This is defined by the extreme value theorem. Hence, for every $r\geq 0$, there exists $n(r)\in\mb{Z}_{\ge 0}$ such that $T_{i}\subset B(r, 0)$ if and only if $1\le i\leq n(r)$.

\begin{lemma}\label{lem:nr-bd}
$4\pi (r-\gamma)^{3}/3\leq n(r)\leq 4\pi r^{3}/3$ for every $r\geq \gamma$.
\end{lemma}
\begin{proof}
The upper-bound is clear as $n(r)$ copies of $T$ with volume $1$ lies in $B(r, 0)$ with volume $\frac{4}{3}\pi r^{3}$. 
For the lower bound, suppose $p$ does not lie in $T_i$ for all $1\le i \le n(r)$. Then, it lies in $T_j$ for some $j>n(r)$, so $T_j$ is not contained in $B(r, 0)$. Hence, there is some $q\in T_{j}$ such that $\Vert q\Vert> r$. Then, as $p, q\in T_j$, by the triangle inequality and definition of diameter $\gamma$,
\[ \Vert p\Vert \geq \Vert q\Vert -\Vert q-p\Vert > r-\gamma. \]
Hence, $T_1, \dots , T_{n(r)}$ together cover $B(r-\gamma, 0)$. This establishes the lower-bound for $n(r)$.
\end{proof}
Let $C_{1}, \dots , C_{s}$ be all possible combinations of dihedral angles of $T$ that give $\pi$ or $2\pi$. 
For each edge $e$ of $T$, let $C_{i}(e)\in \mb{Z}_{\ge 0}$ denote the coefficient of the dihedral angle at $e$ in the combination $C_{i}$. In particular, $C_i$ involves $e$ if and only if $C_i(e)>0$ by Definition~\ref{def:gen-def}.

Let $L_{i}(r)$ be the set of points $p$ in space that lies on some edge $e$ of some copy $T_{j}$ for $1\leq j\leq n(r)$ such that the $2\pi$ angle around $e$ at point $p$ is made up of combination $C_{i}$ (with possibly a face contributing $\pi$) and it involves dihedral angles from $\lbrace T_{j}: 1\leq j\leq n(r)\rbrace$ only. 
Note, there may be some points on edges whose combination involves some $T_j$ with $j\le n(r)$ and some other $T_J$ with $J >n(r)$, so they do not lie in any $L_{i}(r)$. We collect them into
\[ L_{0}(r) = \left(\bigcup_{j=1}^{n(r)}\bigcup_{\text{edge $e$ of $T_j$}} e\right)\Big\bs \left(\bigcup_{i=1}^{s}L_{i}(r)\right).\]
For every $r\geq 0$, we have $n(r)$ finitely many copies of $T$ to consider. Hence, for each $0\leq i\leq n$, $L_{i}(r)$ is a union of finitely many segments in space, so define lengths $\ell_{i}(r)=\len(L_{i}(r))$. As vertices of $T_j$ contribute trivial length, so we need not consider them.
\begin{lemma}\label{lem:l0r-bd}
$\ell_0(r) = O(r^2)$ for any fixed tiling $(T_i)_i$ of some shape $T\in \msc{M}$.
\end{lemma}
\begin{proof}
Any point $p\in L_{0}(r)$ must be on the edges of two copies $T_j$ and $T_J$ where $j\leq n(r)<J$. Note that $T_j\subset B(r, 0)$ and $T_J\not\subset B(r, 0)$. By the triangle inequality and definition of $\gamma$,
\[ r-\gamma < R_J - \gamma \le \Vert p\Vert \le R_j \le r.\]
Now, $T_j\not\subset B(r-\gamma, 0)$, so $j >n(r-\gamma)$. Thus, $L_0(r)$ is covered by the edges of $T_j$ for $n(r-\gamma) < j \le n(r)$. The total edge length of $T$ is trivally at most $6\gamma$, so by Lemma~\ref{lem:nr-bd},
\[ \ell_0(r) \le \left(n(r)-n(r-\gamma)\right)\cdot 6\gamma \le \left(\frac{4}{3}\pi r^{3}-\frac{4}{3}\pi (r-2\gamma)^{3}\right)\cdot 6\gamma = O(r^2). \qedhere\]
\end{proof}
Now, we prove the following crucial observation. For convenience, define
\begin{equation}\label{E:def-LD}
    \lambda_{i}(r)=\frac{\ell_{i}(r)}{n(r)}\quad\text{and}\quad D_{i}(e)=\frac{C_{i}(e)}{\len (e)}.
\end{equation}
\begin{proposition}\label{prop:lp-key}
Fix any tiling $\Ti$ of shape $T\in \msc{M}$. For every edge $e$ of $T$,
\begin{equation}\label{E:lp-key}
\lim_{r\to\infty} \sum_{i=1}^{s}\lambda_{i}(r)D_{i}(e)= 1.
\end{equation}

\end{proposition}
\begin{proof}
For each edge $e$ of $T$, the total length of copies of edge $e$ in $T_1, \dots, T_{n(r)}$ is $n(r)\len (e)$. 

We can also compute this in a different way. Each $C_{i}$ is used in a total length of $\ell_{i}(r)$, so $\ell_{i}(r)C_{i}(e)$ counts the total length of copies of $e$ involved in combination $C_{i}$. 
When we take this sum over $1\leq i\leq s$, 
we get $n(r)\len (e)$ except where copies of $e$ intersect $L_{0}$. The total length of segments in $L_{0}$ is $\ell_{0}$, but we do not know the coefficient ``$C_{0}(e)$'' of $e$ used in whichever dihedral angle combination $C_i(e)$ it was used. 
We know that if we consider $T_j$ for all $j\in \mb{N}$, it should have multiplicity $C_i(e)$ for some $i$, but some of those contributions may be from $T_j$ for $j > n(r)$. Nevertheless, we can bound this below by zero and above by $C_{i}(e)$. For each $e\in \E$, we can then uniformly bound above by the maximum $C_{i}(e)$ over all $i$, i.e.
\[ 0\leq  n(r)\len (e) - \sum_{i=1}^{s}\ell_{i}(r)C_{i}(e) \leq \ell_{0}(r)\max_{1\leq i\leq s} C_{i}(e).\]
Now, divide by $n(r)\len(e)$ to obtain that for every edge $e$,
\[ 0\le  1 - \sum_{i=1}^{s}\frac{\ell_{i}(r)C_{i}(e)}{n(r)\len(e)} \leq \frac{\ell_{0}(r)}{n(r)} \cdot \frac{\max_{1\leq i\leq s}C_{i}(e)}{\len(e)} = \frac{O(r^2)}{O(r^3)} \to 0\]
by Lemmas~\ref{lem:nr-bd} and \ref{lem:l0r-bd}, so it is $0$ as $r\to\infty$. Finally, plug in \eqref{E:def-LD}.
\end{proof}
Instead of the limit statement, we will use a weaker existence version of Proposition~\ref{prop:lp-key} and apply linear program duality to develop a new criterion. Say $\x\succeq\y$ if $x_{i}\geq y_{i}$ for every component $i$. Similarly, define $\succ, \prec$, and $\preceq$. Let $\mbf{0}$ denote the origin $\mbf{1}$ denote the all-ones vector with dimension clear from context. For emphasis, we bold all vectors. Let
\[ D =\begin{pmatrix}
D_{1}(e_{12}) & D_{2} (e_{12}) & \dots & D_{s}(e_{12}) \\
D_{1}(e_{34}) & D_{2} (e_{34}) & \dots & D_{s}(e_{34}) \\
D_{1}(e_{13}) & D_{2} (e_{13}) & \dots & D_{s}(e_{13}) \\
D_{1}(e_{24}) & D_{2} (e_{24}) & \dots & D_{s}(e_{24}) \\
D_{1}(e_{14}) & D_{2} (e_{14}) & \dots & D_{s}(e_{14}) \\
D_{1}(e_{23}) & D_{2} (e_{23}) & \dots & D_{s}(e_{23})
\end{pmatrix}\in \R_{\geq 0}^{6\times s}.\]
\begin{lemma}\label{lem:pre-lp-lem}
If $T$ tiles space, then there exists $\x\in \mb{R}^s$ such that $D\x=\mbf{1}$ and $\x \succeq \mbf{0}$.
\end{lemma}
\begin{proof}
For sufficiently large $r$, the left hand side of \eqref{E:lp-key} is at most $2$. For every $i$, $C_{i}(e_i)\ge 1$ for some edge $e_i$ of $T$. As the terms on the left hand side are non-negative, we can bound each $\lambda_i(r)\le 2/D_i(e_i)$ independently of $r$, for sufficiently large $r$. Therefore,
\[(\lambda_{1}(r), \dots , \lambda_{s}(r))\in \prod_{i=1}^{s}\left[0, \frac{2}{D_{i}(e_i)}\right].\]
By the Bolzano-Weiestrass theorem, the left hand side converges to some $\x = (\Lambda_{i})_{i=1}^{s} \succeq \mbf{0}$ for $r$ along some sequence $(r_{k})_{k}\to \infty$ as $k\to\infty$. By Proposition~\ref{prop:lp-key}, for every edge $e$
\[ (D\x)(e) = \sum_{i=1}^{s}\Lambda_{i}D_{i}(e) =\lim_{k\to\infty}\sum_{i=1}^{s}\lambda_{i}(r_{k})D_{i}(e) = \lim_{r\to \infty}\sum_{i=1}^{s}\lambda_{i}(r)D_{i}(e) = 1.\qedhere\]
\end{proof}
\begin{corollary}\label{cor:crtl}
$T\in\msc{M}$ does not tile if there exists $ \y\in \R^{6}$ such that 
$\1^\intercal \y <0$ and $D^\intercal \y\succeq0$.
\end{corollary}
\begin{proof}
By Lemma~\ref{lem:pre-lp-lem}, $D\x = \1$ for some $\x\succeq 0$ if $T$ tiles space. Apply Farkas' Lemma.
\end{proof}
\begin{remark}
In \cite{criteriapy}, Corollary~\ref{cor:crtl} is implemented as CRTL by computing the edge lengths, generating all combinations $C_i$, solving the linear program $ \min\{\1^\intercal \y: D^\intercal \y \succeq \mbf{0}\}$ with SciPy, and checking whether the minimum is negative. Edge lengths of a tetrahedron can be computed from its dihedral angles via \cite{wirth2014relations}, as implemented in \cite{tetrapy}.
\end{remark}
\section{Applying the Criteria to Rational Tetrahedra}\label{sec:app}
In this section, we apply our three criteria (Propositions~\ref{prop:crtn} and \ref{prop:crtf}, and Corollary~\ref{cor:crtl}) to prove Lemma~\ref{lem:negative}, thereby proving Theorems~\ref{thm:main-family} and \ref{thm:main-sporadic}.
They are implemented in \cite{criteriapy} with some computation in \cite{tetrapy}. We discuss the results when each criterion is applied.
\subsection{Application to $\msc{S}$}\label{sec:app-s}
Since $\msc{S}$ contains only finitely many cases, we check immediately that Proposition~\ref{prop:crtf} and Proposition~\ref{prop:crtn} rule out $17$ of the $59$ sporadic tetrahedra in $\msc{S}$. The remaining ones are the $40$ listed as $\msc{A}$ in Theorem~\ref{thm:main-sporadic} and the following two:
\begin{equation}\label{E:extra-2}
\left(
\frac{ 3 }{ 20 }
, 
\frac{ 11 }{ 20 }
, 
\frac{ 11 }{ 20 }
, 
\frac{ 1 }{ 4 }
, 
\frac{ 1 }{ 3 }
, 
\frac{ 2 }{ 3 }
\right)
\qquad\text{and}\qquad
\left(
\frac{ 11 }{ 60 }
, 
\frac{ 31 }{ 60 }
, 
\frac{ 11 }{ 20 }
, 
\frac{ 1 }{ 4 }
, 
\frac{ 3 }{ 10 }
, 
\frac{ 7 }{ 10 }
\right).
\end{equation}
These two are ruled out by checking Corollary~\ref{cor:crtl}. This concludes the proof of Theorem~\ref{thm:main-sporadic}.
\subsection{Application to $\msc{F}_2$}\label{sec:app-f}
We begin by narrowing down $\msc{F}_2$ to finitely many cases to check.
\begin{definition}
Say $x$ and $T_{x}$ in the second family are \emph{general} if every combination of the dihedral angles of $T_x$ that gives $\pi$ holds for every $x$. Say $x$ and $T_{x}$ are \emph{specific} otherwise.
\end{definition}
\begin{lemma}\label{lem:Tx-n}
If $T_x$ is general, the only combination of dihedral angles that gives $\pi$ or $2\pi$ are
\begin{equation}\label{E:Tx-comb}
\begin{aligned}
\left(\frac{5\pi}{6}-x\right)+\left(\frac{\pi}{6}+x\right)&=\pi, \\
2\left(\frac{5\pi}{6}-x\right)+2\left(\frac{\pi}{6}+x\right)&=2\pi, \\
3\left(\frac{2\pi}{3}-x\right)+3x &=2\pi.
\end{aligned}
\end{equation}
\end{lemma}
\begin{proof}
Consider the following linear combination of dihedral angles: for $a, b, c, d\in\mb{Z}_{\ge 0}$,
\[ a\left(\frac{5\pi}{6}-x\right)+b\left(\frac{\pi}{6}+x\right)+c\left(\frac{2\pi}{3}-x\right)+dx = \frac{5a+b+4c}{6}\pi + (b+d-a-c)x\in\lbrace \pi, 2\pi\rbrace,\]
so it must be the $x$ coefficient $b+d-a-c=0$ and $5a+b+4c\in \lbrace 6, 12\rbrace$. We check this directly, as $a, b, c\in \mb{Z}_{\ge 0}$. It is also implemented in \cite{criteriapy} as GEN.
\end{proof}
\begin{remark}
The proof also gives us a way to generate specific members of $\msc{F}_2$. There, coefficient $b+d-a-c\neq 0$. Since $x\in (\pi/6, \pi/3]$, every dihedral angle is at least $\pi/{6}$ and the sum is at most $2\pi$, so coefficients $0\leq a, b, c, d\leq 12$. Checking them, we see that there are $23$ specific tetrahedra given in Appendix~\ref{sec:spe}. This is implemented in \cite{criteriapy} as SPE.
\end{remark}
Therefore, by Proposition~\ref{prop:crtn}(2), general members of $\msc{F}_2$ do not have a non-face-to-face tiling. We also see that $T_{\pi/3}$ is specific, as it has more dihedral combinations than those in \eqref{E:Tx-comb}.
We now check the conditions of Proposition~\ref{prop:crtf} for all tetrahedra in $\msc{F}_2$ other than $T_{\pi/3}$.
\begin{lemma}\label{lem:Tx-f}
The following statements hold for all $x\in (\pi/6, \pi/3)\cap \Q\pi$:
\begin{enumerate}
    \item $T_{x}$ is parallelogram-like;
    \item $T_{x}$ has a dihedral angle not dividing $2\pi$.
\end{enumerate}
\end{lemma}
\begin{proof}
(1) is checked computationally in Appendix~\ref{sec:d12d24}.
Suppose (2) does not hold, then $x\mid 2\pi$, so $x={2\pi}/{n}$ for some $n\in \N$. Now, ${2}/{3}-{2}/{n}\mid 2$ or ${n-3}/{3n}\mid 1$ gives $n\in \lbrace 4, 6, 12\rbrace$ or $x\in \lbrace \pi/2, \pi/3, \pi/6\rbrace$. In range $(\pi/6, \pi/3]$, the only solution is $x=\pi/3$.
\end{proof}
Hence, by Proposition~\ref{prop:crtf}, no member of $\msc{F}_2\setminus\{T_{\pi/3}\}$ has a non-face-to-face tiling. Together with Proposition~\ref{prop:crtn}(2), we get that no general member of $\msc{F}_2$ tiles space. We check that Proposition~\ref{prop:crtn} applies for all the $23$ specific tetrahedra in Appendix~\ref{sec:spe} except for $T_{\pi/3}$ and $T_{\pi/4}$, and that Corollary~\ref{cor:crtl} applies to show $T_{\pi/4}$ does not tile, so we have the following.
\begin{lemma}\label{lem:F2-res}
The only tetrahedron in $\msc{F}_2$ that tiles space is $T_{\pi/3}$.
\end{lemma}
This concludes the proof of Lemma~\ref{lem:negative} and Theorem~\ref{thm:main-family}.
\begin{appendices}
\section{List of Specific Tetrahedra in One-Parameter Family $\mathscr{F}_2$}\label{sec:spe}
Listed by their dihedral angles $\left(\alpha_{12}, \alpha_{34}, \alpha_{13}, \alpha_{24}, \alpha_{14}, \alpha_{23}\right)$ as multiples of $\pi$, the $23$ specific tetrahedra in $\msc{F}_2$ are
{\tiny \begin{multline*}
\Bigg\lbrace
\left( \frac{ 1 }{ 2 } , \frac{ 1 }{ 2 } , \frac{ 1 }{ 3 } , \frac{ 1 }{ 3 } , \frac{ 1 }{ 3 } , \frac{ 1 }{ 3 } \right), \left( \frac{ 7 }{ 12 } , \frac{ 5 }{ 12 } , \frac{ 5 }{ 12 } , \frac{ 5 }{ 12 } , \frac{ 1 }{ 4 } , \frac{ 1 }{ 4 } \right), \left( \frac{ 19 }{ 30 } , \frac{ 11 }{ 30 } , \frac{ 7 }{ 15 } , \frac{ 7 }{ 15 } , \frac{ 1 }{ 5 } , \frac{ 1 }{ 5 } \right), \left( \frac{ 23 }{ 42 } , \frac{ 19 }{ 42 } , \frac{ 8 }{ 21 } , \frac{ 8 }{ 21 } , \frac{ 2 }{ 7 } , \frac{ 2 }{ 7 } \right),  \left( \frac{ 11 }{ 18 } , \frac{ 7 }{ 18 } , \frac{ 4 }{ 9 } , \frac{ 4 }{ 9 } , \frac{ 2 }{ 9 } , \frac{ 2 }{ 9 } \right),
\\ \left( \frac{ 43 }{ 66 } , \frac{ 23 }{ 66 } , \frac{ 16 }{ 33 } , \frac{ 16 }{ 33 } , \frac{ 2 }{ 11 } , \frac{ 2 }{ 11 } \right), \left( \frac{ 17 }{ 30 } , \frac{ 13 }{ 30 } , \frac{ 2 }{ 5 } , \frac{ 2 }{ 5 } , \frac{ 4 }{ 15 } , \frac{ 4 }{ 15 } \right), \left( \frac{ 9 }{ 14 } , \frac{ 5 }{ 14 } , \frac{ 10 }{ 21 } , \frac{ 10 }{ 21 } , \frac{ 4 }{ 21 } , \frac{ 4 }{ 21 } \right),  \left( \frac{ 5 }{ 9 } , \frac{ 4 }{ 9 } , \frac{ 7 }{ 18 } , \frac{ 7 }{ 18 } , \frac{ 5 }{ 18 } , \frac{ 5 }{ 18 } \right), 
 \left( \frac{ 5 }{ 8 } , \frac{ 3 }{ 8 } , \frac{ 11 }{ 24 } , \frac{ 11 }{ 24 } , \frac{ 5 }{ 24 } , \frac{ 5 }{ 24 } \right), 
 \\ \left( \frac{ 19 }{ 36 } , \frac{ 17 }{ 36 } , \frac{ 13 }{ 36 } , \frac{ 13 }{ 36 } , \frac{ 11 }{ 36 } , \frac{ 11 }{ 36 } \right), \left( \frac{ 4 }{ 7 } , \frac{ 3 }{ 7 } , \frac{ 17 }{ 42 } , \frac{ 17 }{ 42 } , \frac{ 11 }{ 42 } , \frac{ 11 }{ 42 } \right),  \left( \frac{ 29 }{ 48 } , \frac{ 19 }{ 48 } , \frac{ 7 }{ 16 } , \frac{ 7 }{ 16 } , \frac{ 11 }{ 48 } , \frac{ 11 }{ 48 } \right),
\left( \frac{ 17 }{ 27 } , \frac{ 10 }{ 27 } , \frac{ 25 }{ 54 } , \frac{ 25 }{ 54 } , \frac{ 11 }{ 54 } , \frac{ 11 }{ 54 } \right),\\
\left( \frac{ 13 }{ 20 } , \frac{ 7 }{ 20 } , \frac{ 29 }{ 60 } , \frac{ 29 }{ 60 } , \frac{ 11 }{ 60 } , \frac{ 11 }{ 60 } \right), \left( \frac{ 13 }{ 24 } , \frac{ 11 }{ 24 } , \frac{ 3 }{ 8 } , \frac{ 3 }{ 8 } , \frac{ 7 }{ 24 } , \frac{ 7 }{ 24 } \right),  \left( \frac{ 3 }{ 5 } , \frac{ 2 }{ 5 } , \frac{ 13 }{ 30 } , \frac{ 13 }{ 30 } , \frac{ 7 }{ 30 } , \frac{ 7 }{ 30 } \right), \left( \frac{ 23 }{ 36 } , \frac{ 13 }{ 36 } , \frac{ 17 }{ 36 } , \frac{ 17 }{ 36 } , \frac{ 7 }{ 36 } , \frac{ 7 }{ 36 } \right), \left( \frac{ 25 }{ 42 } , \frac{ 17 }{ 42 } , \frac{ 3 }{ 7 } , \frac{ 3 }{ 7 } , \frac{ 5 }{ 21 } , \frac{ 5 }{ 21 } \right),\\ \left( \frac{ 35 }{ 54 } , \frac{ 19 }{ 54 } , \frac{ 13 }{ 27 } , \frac{ 13 }{ 27 } , \frac{ 5 }{ 27 } , \frac{ 5 }{ 27 } \right), \left( \frac{ 8 }{ 15 } , \frac{ 7 }{ 15 } , \frac{ 11 }{ 30 } , \frac{ 11 }{ 30 } , \frac{ 3 }{ 10 } , \frac{ 3 }{ 10 } \right), \left( \frac{ 13 }{ 21 } , \frac{ 8 }{ 21 } , \frac{ 19 }{ 42 } , \frac{ 19 }{ 42 } , \frac{ 3 }{ 14 } , \frac{ 3 }{ 14 } \right), \left( \frac{ 31 }{ 48 } , \frac{ 17 }{ 48 } , \frac{ 23 }{ 48 } , \frac{ 23 }{ 48 } , \frac{ 3 }{ 16 } , \frac{ 3 }{ 16 } \right) 
\Bigg\rbrace.
\end{multline*}}
\section{Proof that Tetrahedra in the First Family are Parallelogram-like}\label{sec:d12d24}
We show that $e_{24}=e_{13}$ is the only pair of equal edge for any $T_x$. Using \cite{tetrapy}, we compute the edge-lengths of $T_{x}$ as functions of $x$.
\begin{align*}
e_{14} & =1=e_{23},\\
e_{12} & =\sqrt{\frac{1-\frac{\left(-\cos\left(x\right)\cos\left(x+\frac{\pi}{3}\right)+\cos\left(x+\frac{\pi}{6}\right)\right)^{2}}{\sin^{2}\left(x\right)\sin^{2}\left(x+\frac{\pi}{3}\right)}}{1-\frac{\left(\cos\left(x\right)+\cos\left(x+\frac{\pi}{6}\right)\cos\left(x+\frac{\pi}{3}\right)\right)^{2}}{\sin^{2}\left(x+\frac{\pi}{6}\right)\sin^{2}\left(x+\frac{\pi}{3}\right)}}},\\
e_{24} &=\sqrt{\frac{1-\frac{\left(-\cos\left(x\right)\cos\left(x+\frac{\pi}{6}\right)-\cos\left(x+\frac{\pi}{3}\right)\right)^{2}}{\sin^{2}\left(x\right)\sin^{2}\left(x+\frac{\pi}{6}\right)}}{1-\frac{\left(\cos\left(x\right)+\cos\left(x+\frac{\pi}{6}\right)\cos\left(x+\frac{\pi}{3}\right)\right)^{2}}{\sin^{2}\left(x+\frac{\pi}{6}\right)\sin^{2}\left(x+\frac{\pi}{3}\right)}}}=e_{13},\\
e_{34} &=\sqrt{\frac{1-\frac{\left(-\cos\left(x\right)\cos\left(x+\frac{\pi}{3}\right)-\cos\left(x+\frac{\pi}{6}\right)\right)^{2}}{\sin^{2}\left(x\right)\sin^{2}\left(x+\frac{\pi}{3}\right)}}{1-\frac{\left(\cos\left(x\right)-\cos\left(x+\frac{\pi}{6}\right)\cos\left(x+\frac{\pi}{3}\right)\right)^{2}}{\sin^{2}\left(x+\frac{\pi}{6}\right)\sin^{2}\left(x+\frac{\pi}{3}\right)}}}.
\end{align*}

We solve for pairwise intersection points for real values $x\in (\pi/6, \pi/3]$. We see that the only solutions are $e_{12}=e_{24}$. Note, on the interval $\sin x, \sin (x+\pi/3), \sin (x+\pi/6), -\cos\left(x\right)\cos\left(x+\frac{\pi}{3}\right)+\cos\left(x+\frac{\pi}{6}\right),$ and $-\cos\left(x\right)\cos\left(x+\frac{\pi}{6}\right)-\cos\left(x+\frac{\pi}{3}\right)$ are nonnegative. We solve $e_{12}=e_{24}$.
{\small \begin{align*}
e_{12}=\sqrt{\frac{1-\frac{\left(-\cos\left(x\right)\cos\left(x+\frac{\pi}{3}\right)+\cos\left(x+\frac{\pi}{6}\right)\right)^{2}}{\sin^{2}\left(x\right)\sin^{2}\left(x+\frac{\pi}{3}\right)}}{1-\frac{\left(\cos\left(x\right)+\cos\left(x+\frac{\pi}{6}\right)\cos\left(x+\frac{\pi}{3}\right)\right)^{2}}{\sin^{2}\left(x+\frac{\pi}{6}\right)\sin^{2}\left(x+\frac{\pi}{3}\right)}}} 
&=\sqrt{\frac{1-\frac{\left(-\cos\left(x\right)\cos\left(x+\frac{\pi}{6}\right)-\cos\left(x+\frac{\pi}{3}\right)\right)^{2}}{\sin^{2}\left(x\right)\sin^{2}\left(x+\frac{\pi}{6}\right)}}{1-\frac{\left(\cos\left(x\right)+\cos\left(x+\frac{\pi}{6}\right)\cos\left(x+\frac{\pi}{3}\right)\right)^{2}}{\sin^{2}\left(x+\frac{\pi}{6}\right)\sin^{2}\left(x+\frac{\pi}{3}\right)}}}=e_{24} \\
\frac{\left(-\cos\left(x\right)\cos\left(x+\frac{\pi}{3}\right)+\cos\left(x+\frac{\pi}{6}\right)\right)^{2}}{\sin^{2}\left(x\right)\sin^{2}\left(x+\frac{\pi}{3}\right)} & =\frac{\left(-\cos\left(x\right)\cos\left(x+\frac{\pi}{6}\right)-\cos\left(x+\frac{\pi}{3}\right)\right)^{2}}{\sin^{2}\left(x\right)\sin^{2}\left(x+\frac{\pi}{6}\right)}\\
\sin\left(x+\frac{\pi}{6}\right)\left(-\cos\left(x\right)\cos\left(x+\frac{\pi}{3}\right)+\cos\left(x+\frac{\pi}{6}\right)\right) & = \sin\left(x+\frac{\pi}{3}\right)\left(-\cos\left(x\right)\cos\left(x+\frac{\pi}{6}\right)-\cos\left(x+\frac{\pi}{3}\right)\right) \\
\cos x\sin \left(\left(x+\frac{\pi}{6}\right)-\left(x+\frac{\pi}{3}\right)\right) & =\frac{1}{2}\sin \left(2\left(x+\frac{\pi}{6}\right)\right)+\frac{1}{2}\sin \left(2\left(x+\frac{\pi}{3}\right)\right)\\
-\frac{1}{2}\cos x & = \frac{\sqrt{3}}{2}\cos (2x) \\
\cos x & = \frac{-1\pm 5}{4\sqrt{3}}.
\end{align*}}
By checking the range cosine in the interval, it must be $\cos x={1}/{\sqrt{3}}$. Since $\cos \left(2x\right) =-1/3$, $x=\arccos\left(1/\sqrt{3}\right)\not\in \Q\pi$ by Niven's Theorem, so $T_x\not\in \msc{F}_2$ need not be considered. 

\end{appendices}
\bibliographystyle{alpha}
\bibliography{ref}

\end{document}